\documentclass[12pt]{amsart}
\usepackage{amsmath}
\usepackage{amsfonts}
\usepackage{amssymb}
\usepackage{amscd}
\usepackage{fancyvrb}
\usepackage[all]{xypic}
\usepackage{mathrsfs}

\usepackage{graphicx}
\usepackage{epsfig}
\usepackage{pstricks}

\newtheorem{theorem}{Theorem}

\newtheorem{corollary}[theorem]{Corollary}
\newtheorem{lemma}[theorem]{Lemma}
\newtheorem{definition}[theorem]{Definition}

\newtheorem{remark}[theorem]{Remark}
\newtheorem*{conjecture}{Conjecture}

\newcommand{\dbZ}{\mathbb{Z}}
\newcommand{\dbS}{\mathbb{S}}

\newcommand{\dbR}{\mathbb{R}}

\newcommand{\dbK}{\mathbb{K}}
\newcommand{\dbL}{\mathbb{L}}

\newcommand{\calL}{{\mathscr L}}

\newcommand{\spec}{{\mathscr L}_*()}

\newcommand{\func}[3]{#1:#2\to#3}

\begin{document}

\title[$K$-FIC and $L$-FIC for Braid groups]{\bf THE $K$ and $L$ theoretic Farrell-Jones isomorphism conjecture for braid groups}

\author{Daniel Juan-Pineda}
\address{Centro de Ciencias Matem\'aticas,
Universidad Nacional Aut\'onoma de M\'exico, Campus Morelia, Apartado Postal 61-3 (Xangari), Morelia,
 Mi\-cho\-a\-c\'an, MEXICO 58089}
\email{daniel@matmor.unam.mx}
\thanks{We acknolwedge support from grants from DGAPA-UNAM and CONACyT-M\'exico}

\author{Luis Jorge S\'anchez Salda\~na}
\address{Centro de Ciencias Matem\'aticas, Universidad Nacional Aut\'onoma de M\'exico, Apartado Postal 61-3 (Xangari), Morelia,
 Mi\-cho\-a\-c\'an, MEXICO 58089}
 \thanks{The second Author would like to thank Prof. D. Juan Pineda for his encouragement and acknowledges support from a CONACyT-M\'exico graduate scholarship}

\email{luisjorge@matmor.unam.mx}

\begin{abstract}
We prove the $K$ and $L$ theoretic versions of the Fibered Isomorphism Conjecture of F. T. Farrell and L. E. Jones for braid groups on a surface.

\end{abstract}

\maketitle

\section{Introduction}

Aravinda, Farrell, and Roushon in \cite{braids} showed that the Whitehead group of  the classical pure braid groups vanishes. Later on, in \cite{fullbraids}, Farrell and Roushon extended this result to the full braid groups. Computations of the Whitehead group of braid groups for the sphere and the projective space were perfomed by D. Juan-Pineda and S. Mill\'an in \cite{DS1,DS2,DS3}. In all cases the key ingredient was to prove the Farrell-Jones isomorphism conjecture for the \textit{Pseudoisotopy} functor, this allows computations for lower algebraic $K$ theory groups. In this note,  we use results by Bartels, L\"uck and Reich \cite{aplications} for word hyperbolic groups and Bartels, L\"uck and Wegner \cite{luckKL,cat(0)} for CAT(0) groups to prove that the algebraic $K$ and $L$ theory versions of the conjecture for braid groups on surfaces.

\section{The Farrell-Jones conjecture and its fibered version.}

In this work we are interested in the formulations of the Farrell-Jones Isomorphism Conjecture, in both, its fibered and nonfibered versions.
Let $\dbK$ be the Davis-L\"uck algebraic $K$ theory functor, see \cite[Section 2]{DL} and let $\dbL$ be the Davis-L\"uck $\calL^{-\infty}$ theory functor, see \cite[Section 2]{DL} and \cite[Section 1.3]{FJ} for the definition of $\calL^{-\infty}$-theory (or $L$ theory for short).  
 The validity of this conjecture allows us, in principle, to approach the algebraic $K$ or $L$ theory groups of the group ring of a given group $G$ from (a) the corresponding $K$ or $L$ theory groups of the virtually cyclic subgroups of $G$ and (b) homological information. More precisely, let $\spec$ be any of the functors $\dbK$ or $\dbL$.
 
\begin{conjecture}
\textbf{(Farrell-Jones Isomorphism Conjecture, IC)} Let $G$ be a discrete group. Then for all $n\in \mathbb{Z}$ the assembly map 
\begin{equation}
\func{A_{Vcyc}}{H_n^G(\underline{\underline{E}}G;\spec)}{H_n^G(pt;\spec)}
\end{equation}
 induced by the projection $\underline{\underline{E}}G\rightarrow pt$ is an isomorphism,
 where the groups  $H^G_*(--;\spec)$ build up a suitable equivariant homology theory with local coefficients the functor $\spec$, and $\underline{\underline{E}}G$ is a model for the classifying space for actions with isotropy in the family of virtually cyclic subgroups of $G$.
\end{conjecture}

A generalization of the Farrell-Jones conjecture is what  is known as the \textit{Fibered Isomorphism Conjecture} (FIC). This generalization has better hereditary properties (see \cite{aplications} and \cite{induction}).

\begin{definition}
Given a homomorphism of groups $\func{\varphi}{K}{G}$ and a family of subgroups $\mathcal{F}$ of $G$ closed under conjugation and finite intersections, we define the induced family $$\varphi^*\mathcal{F}=\{H\leq K| \varphi(H)\in \mathcal{F}\}.$$
If $K$ is a subgroup of $G$ and $\varphi$ is the inclusion we denote $\varphi^*\mathcal{F}=\mathcal{F} \cap K$.
\end{definition}

In this note, we will use FIC to refer to the Fibered Isomorphism Conjecture with respect to the family $Vcyc$  (of virtually cyclic subgroups), for either $K$- or$ L$- theory.

\begin{definition}
Let $G$ be a group and let $\mathcal{F}$ be a family of subgroups of $G$. We say that the pair $(G,\mathcal{F})$ satisfies the \textbf{Fibered  Farrell-Jones Isomorphism Conjecture (FIC)} if for all group homomorphisms  $\func{\varphi}{K}{G}$ the pair $(K,\varphi^*\mathcal{F})$ satisfies that 
$$\func{A_{\varphi^*\mathcal{F}}}{H_n^K(E_{\varphi^*\mathcal{F}}K;\spec)}{H_n^K(pt;\spec)}
$$
 is an isomorphism for all $n\in\mathbb{Z}$.
\end{definition}

One of the most interesting hereditary properties of FIC  is the so called \textit{Transitivity Principle} \cite[Theorem 2.4]{aplications}:

\begin{theorem}\label{transitivity}
Let $G$ be a group and let $\mathcal{F}\subset \mathcal{G}$ be two families of subgroups of $G$. Assume that $N\in \mathbb{Z}\cup \{\infty\}$. Suppose that for every element $H\in \mathcal{G}$ the group $H$ satisfies FIC for the family $\mathcal{F}\cap H$ for all $n\leq N$.
Then $(G,\mathcal{G})$ satifies FIC for all $n\leq N$ if and only if $(G,\mathcal{F})$ satisfies FIC for all $n\leq N$.
\end{theorem}

Let $G$ be a group, we denote $Vcyc(G)$ for the family or virtually cyclic subgroups of $G$. The next two theorems are fundamental for  later sections, see \cite[lemma 2.8]{aplications}.

\begin{theorem}\label{extensiones}
Let $\func{f}{G}{Q}$ be a surjective homomorphism of groups. Suppose that $(Q, Vcyc(Q))$ satisfies FIC and for all $f^{-1}(H)$ with $H\in Vcyc(Q)$, $(f^{-1}(H), Vcyc(f^{-1}(H)))$    FIC is true. Then $(G,Vcyc(G))$ satisfies FIC.
\end{theorem}

\begin{theorem}\label{subgroupfic}
If $G$ satisfies FIC then every subgroup of $G$  satisfies FIC as well.
\end{theorem}

%Suppose that $G$ is a torsion free group that satifies FIC. In this case the family $Vcyc$ of virtually cyclic subgroups reduces to the family of cyclic subgroups  $Cyc$. Applying the Transitivity Principle it can be shown that the following assembly map is an isomorphism:
%$$\func{A}{H_n^G(EG;\calK)}{H_n^G(E_{Cyc}G;\calK)},
%$$
%where $EG$ denotes the usual classifying space for $G$.
%As a consequence, we have the torsion free version of  IC.

%\begin{theorem}\label{fjtorsionfree}
%Let $G$ be a torsion free group. Then $G$ satisfies the Farrell-Jones conjecture if and only if the assembly map 
%$$
%H_n(BG;\calK)\rightarrow K_n(\mathbb{Z}G)
%$$
 %is an isomorphism for all $n\in \mathbb{Z}$.
%\end{theorem}

%%%%%%%%%%%%%%%%%%%%%%%%%%%%%%%%%%%%%%%%%%%%%%%%%%%%%%%%%%%%%%%%%%%%%%%%%%%%%%%%%%%%%%%%%%%%%%%%%%%%%%%%%%%%%%%%%%%%%%%%%%%%%%%%%%%%%%%%%%%%%%%%%%%%

\section{Pure Braid Groups on Aspherical Surfaces.}

\begin{definition}\cite[Definition 1.1]{braids}\label{SPF}
We say that a group $G$ is strongly poly-free if there exists a filtration $1=G_0\subset G_1 \subset \cdots \subset G_n=G$ such that the following conditions hold:
\begin{enumerate}
\item $G_i$ is normal in $G$ for each $i$.
\item $G_{i+1}/G_i$ is a finitely generated free group.
\item  For each $g\in G$  there exists a compact surface $F$ and a diffeomorphism $\func{f}{F}{F}$ such that the action by conjugation of $g$ in $G_{i+1}/G_i$  can be geometrically realized, i.e., the following diagram commutes:
$$
\xymatrix{\pi_1(F) \ar[r]^{f_\#}\ar[d]_{\varphi} & \pi_1(F) \\ G_{i+1}/G_i \ar[r]^{C_g} & G_{i+1}/G_i\ar[u]_{\varphi^{-1}}}
$$
where $\varphi$ is an suitable isomorphism.
\end{enumerate}

\end{definition}

In this situation we say that $G$ has rank lower or equal than $n$.
Now we enunciate some  theorems that will be useful later.

%\begin{lemma}\label{extfic}
%Let $1\rightarrow K\rightarrow G \rightarrow Q \rightarrow 1$ be a group extension. Suppose that $K$ is free and $Q$ is torsion free and satisfies FIC. Then the assembly map 
%$$
%H_n(BG;\calK)\rightarrow K_n(\mathbb{Z}G;\calK)
%$$ 
%is an isomorphism for all $n\in \mathbb{Z}$
%\end{lemma}
%\begin{proof}
%See \cite[Theorem 3.6]{aplications}.
%\end{proof}

\begin{theorem}\label{hiperbolicos}
Every word hyperbolic group satisfies FIC. In particular every finitely generated free group satisfies FIC.
\end{theorem}
\begin{proof}
See \cite[Theorem 1.1]{aplications} for $K$-theory and \cite[Theorem B]{luckKL} for $L$-theory.
\end{proof}

\begin{theorem}\label{catfic}
Every  $CAT(0)$ group satisfies FIC.
\end{theorem}
\begin{proof}
See \cite{cat(0)} for $K$-theory and \cite[Theorem B]{luckKL} for $L$-theory.
\end{proof}

\begin{remark}
Both theorems \ref{hiperbolicos} and \ref{catfic} were proven for a more general version of the Farrell-Jones Isomorphism conjecture stated here, namely they were proven for generalized homology theories with coefficients in any additive category, these  versions imply the one given here and also Theorems \ref{transitivity},\ref{extensiones} and \ref{subgroupfic}, see \cite{BR}. 
\end{remark}

\begin{theorem}
Let $M$ be a simply connected complete Riemannian manifold whose sectional curvatures are all nonpositive and let  $G$ be a group. Assume that $G$ acts by isometries on $M$, properly discontinuously and cocompactly, then $G$ is a $CAT(0)$ group. In particular $G$ satisfies FIC.
\end{theorem}
\begin{proof}
 This is the basic example of a $CAT(0)$ group. See \cite{bridson}.
\end{proof}

\begin{lemma}\label{semidirectocat}
Let $F$ be a finitely generated free group and $\func{f}{F}{F}$ be  an automorphism that can be geometrically realized, in the sense of Definition \ref{SPF} (3), then $F\rtimes \mathbb{Z}$, with the action of $\mathbb{Z}$ in $F$ given by $f$, satisfies FIC. 
\end{lemma}

\begin{proof}
It is proven in \cite[Lemma 1.7]{braids} that  under our hypotheses $F\rtimes\dbZ$ is isomorphic to the fundamental group of a compact Riemannian 3-manifold $M$
that supports in its interior a complete Riemannian metric with nonpositive sectional curvature everywhere and the metric is a cylinder near
the boundary, in fact, near the boundary is diffeomorphic to finitely many components of the form $\dbR\times N$ where $N$ is either a torus or a Klein bottle.
Let $DM$ be the double of $M$, by the description of the boundary, we can endow the closed manifold $DM$ with a  Riemannian metric with nonpositive curvature everywhere as in \cite[Proposition 6]{BFJP}. As $DM$ is now compact, it follows that the fundamental group $\pi_1(DM)$ is CAT(0), hence it satisfies FIC. Moreover, 
$F\rtimes \mathbb{Z}\cong \pi_1(M)$ injects into $\pi_1(DM)$ our result follows by Theorem \ref{subgroupfic}. 
\end{proof}

%As a consequence of the above and Theorem \ref{catfic} we have
%\begin{corollary}\label{fic3mfd}
%Let $G=F\rtimes\dbZ$ be a group as in the previous lemma, then $G$ satisfies FIC.
%\end{corollary}

Our main theorem is now the following
\begin{theorem}\label{principal}
Every strongly poly-free group $G$ satisfies FIC.
\end{theorem}
\begin{proof}
Let us proceed by induction on the rank of $G$. The induction base, when $G$ has rank $\leq 1$, is true as in this case $G$ is a finitely generated free group, thus the assertion follows from Theorem \ref{hiperbolicos}.

Assume that strongly poly-free groups of rank $\leq n$ satisfy FIC and let $G$ be a group of rank $\leq n+1$ with $n>0$. We apply Theorem \ref{extensiones} to the surjective homomorphism $\func{q}{G}{G/G_n}$. Observe that $G/G_n$ is a finitely generated free group, hence it satisfies FIC. Let $C\subset G/G_n$ be a  virtually cyclic (and hence cyclic) subgroup, not excluding the possibility of $C$ being $G/G_n$. We have the following cases:
\begin{enumerate}
\item If $C=\{1\}$ then $q^{-1}(C)=G_n$, which is strongly poly-free of rank $\leq n$, hence it satisfies FIC. 

\item Assume now that $C$ is an infinite cyclic subgroup of $G/G_n$. Let 
$$
f:q^{-1}(C)\to \frac{q^{-1}(C)}{G_{n-1}}
$$
be the natural projection and observe that 
$$
\frac{q^{-1}(C)}{G_{n-1}}\cong \frac{G_n\rtimes C}{G_{n-1}}\cong \frac{G_n}{G_{n-1}}\rtimes C.
$$
Moreover, the group $\frac{G_n}{G_{n-1}}\rtimes C$ satisfies the hypotheses of Corollary \ref{semidirectocat} by the condition (3) for SPF groups, thus $\frac{q^{-1}(C)}{G_{n-1}}$ satisfies FIC  and  we  apply Theorem \ref{extensiones} to the homomorphism $f:q^{-1}(C)\to \frac{q^{-1}(C)}{G_{n-1}}$. Let $V\subseteq \frac{q^{-1}(C)}{G_{n-1}}$ be a cyclic subgroup, again we have the following cases:
  \begin{enumerate}
     \item $V=1$ it follows that $f^{-1}(V)=G_{n-1}$ which is an SPF group of rank $\leq n-1$, and it does satisfy FIC by induction.
     \item $V$ is an infinite cyclic subgroup. By the definition of $V$ it fits in a filtration
     $$
     1=G_0\subset G_1\subset\cdots \subset G_{n-1}\subset f^{-1}(V),
     $$
     which gives that $f^{-1}(V)$ is an SPF group of rank $\leq n$, hence it satisfies FIC by induction.
  \end{enumerate}
  It follows that  $q^{-1}(C)$ satisfies FIC and therefore $G$ satisfies FIC.
\end{enumerate}
\end{proof}

We now recall the definition of the pure braid groups on a surface.

\begin{definition}\label{defpuras}
Let $S$ be a surface with boundary and $P_k=\{y_1,...,y_k\}\subset S$ be 
 a finite subset of interior points. Define the configuration space to be $M_n^k(S)=\{(x_1,x_2,...,x_n)| x_i\in S-P_k,\ x_i\neq x_j\ \text{ for }\ i\neq j\}$. The Pure Braid group on $S$ with $n$-strings $B_n(S)$ is by definition $\pi_1(M^0_n(S))$.
\end{definition}

\begin{lemma}
Let $S$ be a surface with boundary. For $n>r\geq 1$ and $k\geq 0$, the projection on the first $r$ coordinates $M_n^k(S^0)\rightarrow M_r^k(S^0)$ is a fibration with fiber $M_{m-r}^{k+r}(S)$, where $S^0=S-\partial S$.
\end{lemma}
\begin{proof}
See \cite[Lemma 1.27]{braidsbook}.
\end{proof}

\begin{lemma}
Suppose that $S=\mathbb{C}$ or that $S$ is a compact surface with nonempty boundary. Then for all $m\geq 0$, $n\geq 1$ the manifold $M^m_n(S)$ is aspherical.
\end{lemma}
\begin{proof}
Consider the fibration $M^m_n(S^0)\rightarrow M^m_1(S^0)=S^0-P_m$ with fiber $M^{m+1}_{n-1}(S^0)$ given by previous lemma. The exact sequence associated to this fibration is as follows 
$$
\cdots \rightarrow \pi_{i+1}(S-P_m)\rightarrow \pi_i(M^{m+1}_{n-1}(S))\rightarrow \pi_i(M^m_n(S))\rightarrow \pi_i(S-P_m)\rightarrow \cdots .
$$
Since $S-P_m$ is aspherical, because the boundary is nonempty, $\pi_i(S-P_m)=0$ for all $i\geq 2$. Hence for all $i\geq 2$,  $\pi_i(M^{m+1}_{n-1}(S))\cong \pi_i(M^m_n(S))$.
An inductive argument shows that for all $i\geq 2$ we have 
$$
\pi_i(M^{m}_{n}(S))\cong \pi_i(M^{m+n-1}_{1}(S))\cong \pi_i(S-P_{m+n-1})=0.
$$
\end{proof}

\begin{theorem}\label{braidspf}
Suppose that $S=\mathbb{C}$ or that $S$ is a compact surface with nonempty boundary different from $\dbS^2$ or $\mathbb{R}P^2$. Then the pure braid group $B_n(S)$ is strongly poly-free of rank $\leq n$ for all $n\geq 1$.
\end{theorem}
\begin{proof}
See \cite[Theorem 3.1]{braids}.
\end{proof}

Recall that the braid groups on $\mathbb{C}$ or on a compact surface other than $\dbS^2$ or $\mathbb{R}P^2$  are torsion free. The above Theorem and Theorem \ref{principal} now gives:

\begin{theorem}\label{trenzaspuras}
Suppose that $S=\mathbb{C}$ or $S$ is a compact surface other than $\dbS^2$ or $\mathbb{R}P^2$. Then the pure braid groups $B_n(S)$ satisfy FIC for all $n\geq 1$.
\end{theorem}

%%%%%%%%%%%%%%%%%%%%%%%%%%%%%%%%%%%%%%%%%%%%%%%%%%%%%%%%%%%%%%%%%%%%%%%%%%%%%%%%%%%%%%%%%%%%%%%%%%%%%%%%%%%%%%%%%%%%%%%%%%%%%%%%%%%%%%%%%%%%%%%%%%%

\section{Full Braid Groups on Aspherical Surfaces.}

The main goal of this section is to prove that any extension of a finite group by an SPF group satisfies FIC. In order to prove this we shall need some results.

\begin{definition}
Let $G$ and $H$ be groups, with $H$ finite. We define de wreath product $G\wr H$ to be the semidirect product $G^{|H|}\rtimes H$, where $G^{|H|}$ is the group of $|H|$-tuples of elements in $G$ indexed by elements in $H$, $H$ acts on $G^{|H|}$  by permuting the coordinates as the action of $H$ on $H$ by right translation.
\end{definition}

Wreath products have been widely studied, the following properties are well known.

\begin{lemma}\label{wreath}
Let $1\rightarrow G\rightarrow \Gamma \rightarrow H\rightarrow 1$  be a group extension with $H$ finite. Then there are injective homomorphisms $\func{\delta}{\Gamma}{G\wr H}$ and $\func{\theta}{G}{G^{|H|}}$ which together with $\func{id}{H}{H}$ define a map to the group extension $1\rightarrow G^{|H|}\rightarrow G\wr H \rightarrow H\rightarrow 1$.
\end{lemma}
\begin{proof}
See \cite[Algebraic Lemma]{fullbraids}.
\end{proof}

The following lemmas contain standard facts about wreath products.

\begin{lemma}\label{wreathfacts}
Let $A$, $B$, $S$ and $H$ be groups with $S$ and $H$ finite.
\begin{enumerate}
\item If $A$ is a subgroup of $B$, then $A\wr H$ is a subgroup of $B\wr H$.
\item $A^{|H|}\wr S$ is a subgroup of $A\wr (H\times S)$.
\end{enumerate}
\end{lemma}

\begin{lemma}\label{semidirectos}\cite[Fact 2.4]{braids}
Let $G$ be a group and $H$ a finite subgroup of $Aut(G)$. We define $G^{|H|}\rtimes \mathbb{Z}$, where the generator of $\mathbb{Z}$ acts on the left factor via $f=\bigoplus_{h\in H} h\in Aut(G^{|H|})$, and $G\rtimes_h \mathbb{Z}$ for each $h\in H$ in the obvious way. Then $G^{|H|}\rtimes \mathbb{Z}$ is a subgroup of $\prod_{h\in H}G\rtimes_h \mathbb{Z}$.
\end{lemma}

%aqu� termina todo lo de wreath products

\begin{theorem}\label{wreathcat}
If $G$ is a $CAT(0)$ group and $H$ is a finite group, then  $G\wr H$ is a  $CAT(0)$ group, and hence satifies FIC.
\end{theorem}
\begin{proof}
 Let $G$ act properly, isometrically and cocompactly on the $CAT(0)$ space $X$, then $G\wr H$ acts properly, isometrically and cocompactly on $X^{|H|}$ with $G^{|H|}$ acting coordinatewise on $X^{|H|}$ and $H$ by permuting the coordinates.
 \end{proof}

\begin{theorem}
Let $G$ be an SPF group, and let $H$ be a finite group. Then $G\wr H$ satisfies FIC.
\end{theorem}
\begin{proof}
Let us proceed by induction on the rank $n$ of $G$. When $n=0$ it follows that  $G=1$ and hence $G\wr H$ is finite, thus is hyperbolic and  it satisfies FIC by Theorem \ref{hiperbolicos}.

Now assume $G$ has rank $\leq n$, where $n>1$, and consider the filtration $1=G_0\subset G_1\subset \cdots \subset G_n=G$ given in the definition of a strongly poly-free group.

Note that $G_1^{|H|}$  is a normal subgroup of $G\wr H$ and hence, we have the group extension
$$
1\rightarrow G_1^{|H|} \rightarrow G\wr H \rightarrow (G/G_1)\wr H \rightarrow 1.
$$
Let $p:G\wr H \rightarrow (G/G_1)\wr H$ denote the above epimorphism. We will apply Theorem \ref{extensiones}, note that this is possible because $G/G_1$ is an SPF of rank $\leq n-1$. Hence $(G/G_1)\wr H$ satisfies FIC by induction hypothesis.

Next, let $S\subset (G/G_1)\wr H$ be a virtually cyclic subgroup. We have to prove that $p^{-1}(S)$ satisfies FIC. There are two cases to consider.

\underline{Case 1}: $S$ is finite. We have an exact sequence $1\rightarrow G_1^{|H|} \rightarrow p^{-1}(S) \rightarrow S \rightarrow 1$. Using Lemma \ref{wreath} and Lemma \ref{wreathfacts}, we get
$$p^{-1}(S)\subset G_1^{|H|}\wr S \subset G_1\wr (H\times S) ,$$ %\subset B\wr (H\times S),$$where $G_1\subset B=\pi_1(M)$,and $M$ is a closed surface of genus $\geq 1$. 
Now by Theorem \ref{wreathcat} and Theorem \ref{subgroupfic}, $p^{-1}(S)$ satisfies FIC.

\underline{Case 2}: $S$ is infinite. $S$ contains a normal subgroup $T$ of finite index such that $T$ is infinite cyclic and $T\subset (G/G_1)^{|H|}.$ In fact, we assume $T=S\cap (G/G_1)^{|H|}$. We have the exact sequence $1\rightarrow p^{-1}(T) \rightarrow p^{-1}(S) \rightarrow S/T \rightarrow 1$. By Lemma \ref{wreath} we get $p^{-1}(S)\subset T_1\wr (S/T)$, where $T_1=p^{-1}(T)$. By Theorem \ref{subgroupfic} it suffices to show that $T_1\wr (S/T)$ satisfies FIC.

Fix $t=(t_h)_{h\in H}\in G^{|H|}$, which maps to a generator of $T$. Note that each $t_h$  acts geometrically on $G_1$. Then, by Lemma \ref{semidirectos} 
$$
T_1=G_1^{|H|}\rtimes_t \mathbb{Z} \subset \prod_{h\in H} (G_1 \rtimes_{t_h} \mathbb{Z}),
$$
now using Definition \ref{SPF}, Lemma \ref{semidirectocat}, and the fact that the finite product of $CAT(0)$-groups is a $CAT(0)$-group we conclude that $\prod_{h\in H} (G_1 \rtimes_{t_h} \mathbb{Z})$ is a $CAT(0)$-group. Hence by Theorems \ref{wreathcat} and \ref{subgroupfic}, $T_1\wr (S/T)$ satisfies FIC. Thus, $p^{-1}(S)$ also satisfies FIC. Thus $G\wr H$ satisfies FIC.
\end{proof}

\begin{corollary}\label{finitebyspf}
Every extension $\Gamma$ of a finite group by an SPF satisfies FIC.
\end{corollary}
\begin{proof}
It is immediate from the previous Theorem and Lemma 21.
\end{proof}

\begin{theorem}
Suppose that $M=\mathbb{C}$ or $M$ is a compact surface other than $\dbS^2$ or $\mathbb{R}P^2$. Let $\Gamma$ be an extension of a finite group $H$ by  $B_n(M)$ for some $n\geq 1$. Then $\Gamma$ satisfies FIC.
\end{theorem}
\begin{proof}
If $M$ has nonempty boundary or $M=\mathbb{C}$ then $B_n(M)$ is SPF by Theorem \ref{braidspf}, so by the previous Corollary the assertion is true.
From now on, we assume that $M$ has empty boundary. By Lemma \ref{wreath} and Theorem \ref{subgroupfic} it suffices to prove that $B_n(M)\wr H$ satisfies FIC.
Considering the fiber bundle projection $p:M^0_n(M)\to M^0_1(M)=M$ with fiber $M^1_{n-1}(M)=M^0_{n-1}(M-pt)$ we have the following short exact sequence
$$
1\rightarrow A=B_{n-1}(M-pt)\rightarrow B_n(M)\stackrel{p}{\rightarrow} \pi_1(M)\rightarrow 1.
$$
Consider the exact sequence 
$$
1 \rightarrow A^{|H|}\rightarrow B_n(M)\wr H \rightarrow \pi_1(M)\wr H \rightarrow 1.
$$
Let $p: B_n(M)\wr H \rightarrow \pi_1(M)\wr H$ be the surjective homomorphism in the above sequence. We proceed to apply Theorem \ref{extensiones} to  $p$. Note that $\pi_1(M)\wr H$ satisfies FIC by Theorem \ref{wreathcat}. Let $S$ be a virtually cyclic subgroup of $\pi_1(M)\wr H$. We claim that $p^{-1}(S)$ contains a SPF of finite index. Let $T=S\cap \pi_1(M)^{|H|}$. Then $p^{-1}(T)$ is of finite index in $p^{-1}(S)$. Now, as $A$ is SPF,  it follows that $A^{|H|}$ is also  SPF by considering the filtration 
$$
1=G_0\subset \cdots G_n=A \subset A\times G_1\subset \cdots A\times A\subset \cdots \subset A^{|H|-1}\times G_{n-1} \subset A^{|H|}
$$
where $1=G_0\subset G_1\subset \cdots G_n=A$ is an SPF structure on $A$. On the other hand, we have an exact sequence $1 \rightarrow A^{|H|} \rightarrow p^{-1}(T) \rightarrow T\rightarrow 1$. Now, from the monodromy action on the pure braid group of $M$ it can be checked that $p^{-1}(T)$ is SPF. The proof now follows from the previous Corollary.
\end{proof}

\begin{definition}
We recall from Definition \ref{defpuras} that 
$$
M_n^0(M)=\{(x_1,x_2,...,x_n)| x_i\in S,\ x_i\neq x_j\ if\ i\neq j\}.
$$
The symmetric group $S_n$ acts on $M_n^0$. We define the Full Braid Group $FB_n(M)$ on a surface $M$ to be $\pi_1(M_n^0/S_n)$.
\end{definition}

It is not difficult to see that we have an exact sequence
$$
1\rightarrow B_n(M)\rightarrow FB_n(M)\rightarrow S_n \rightarrow 1,
$$
hence by our previous Theorem we have the following:

\begin{theorem}
Suppose that $M=\mathbb{C}$ or $M$ is a compact surface other than $\dbS^2$ or $\mathbb{R}P^2$. Then the full braid group $FB_n(M)$ satisfies FIC.
\end{theorem}

%%%%%%%%%%%%%%%%%%%%%%%%%%%%%%%%%%%%%%%%%%%%%%%%%%%%%%%%%%%%%%%%%%%%%%%%%%%%%%%%%%%%%%%%%%%%%%%%%%%%%%%%%%%%%%%%%%%%%%%%%%%%%%%%%%%%%%%%%%%%%%%%%%%%

\section{Braid groups on $\dbS^2$ and $\mathbb{R}P^2$}

\begin{theorem}
The pure braid groups $B_n(\dbS^2)$ satisfy FIC for all $n>0$.
\end{theorem}
\begin{proof}
We have that $B_1(\dbS^2)=B_2(\dbS^2)=0$, and $B_3(\dbS^2)\cong \mathbb{Z}_2$ (see \cite{fadell}), so they satisfy FIC because they are finite. For $n>3$, we consider the fiber bundle $M^0_n(\dbS^2)\rightarrow M^0_3(\dbS^2)$ with fiber $M_{n-3}^3(\dbS^2)\cong M_{n-3}^2(\mathbb{C})$. Applying the long exact  sequence of the corresponding fibration and the fact that $\pi_2(M_3^0(\dbS^2))$ is trivial we get
$$
1\rightarrow \pi_1(M_{n-3}^2(\mathbb{C})) \rightarrow \pi_1(M^0_n(\dbS^2)) \rightarrow \pi_1(M^0_3(\dbS^2)) \rightarrow 1.
$$
Note that $\pi_1(M_{n-3}^2(\mathbb{C}))$ is SPF because it is part of the filtration of $B_{n-3}(\mathbb{C})$, hence by Corollary \ref{finitebyspf} we have that $\pi_1(M^0_n(\dbS^2))=B_n(\dbS^2)$ satisfies FIC.
\end{proof}

\begin{lemma}
Let $1\rightarrow K \rightarrow G \rightarrow Q \rightarrow 1$ be an extension of groups. Suppose that $K$ is virtually cyclic and $Q$ satisfies FIC. Then $G$ satisfies FIC.
\end{lemma}
\begin{proof}
See \cite{aplications} Lemma 3.4.
\end{proof}

\begin{theorem}
The full braid groups $FB_n(\dbS^2)$ satisfy FIC for all $n>0$.
\end{theorem}
\begin{proof}
In \cite{silvia}, Theorem 24 it is proven that $FB_n(\dbS^2)$ fits in an exact sequence 
$$1\rightarrow \Gamma_n \rightarrow FB_n(\dbS^2))/\mathbb{Z}_2 \rightarrow S_n \rightarrow 1$$
where $\Gamma_n$ is SPF. Hence $FB_n(\dbS^2)/\mathbb{Z}_2$ satisfies FIC by Corollary \ref{finitebyspf}.
Now consider the exact sequence 
$$1\rightarrow \mathbb{Z}_2 \rightarrow FB_n(\dbS^2) \rightarrow FB_n(\dbS^2)/\mathbb{Z}_2 \rightarrow 1$$
applying the previous Lemma, we conclude that $FB_n(\dbS^2)$ satisfies FIC.
\end{proof}

\begin{theorem}
The pure braid groups $B_n(\mathbb{R}P^2)$ satisfy FIC for all $n>0$.
\end{theorem}
\begin{proof}
In \cite{buskirk} it is proven that $B_1(\mathbb{R}P^2)=\mathbb{Z}_2$, $B_2(\mathbb{R}P^2)\cong Q_8$ and $B_3(\mathbb{R}P^2)\cong F_2\rtimes Q_8$, where $Q_8$ is the quaternion group with eight elements and $F_2$ is the free group on two generators. Hence $B_1(\mathbb{R}P^2)$ and $B_2(\mathbb{R}P^2)$ satify FIC because they are finite, while $F_2\rtimes Q_8$ does as it is hyperbolic.
In \cite{silvia}, Theorem 27 it is proven that $B_n(\mathbb{R}P^2)$ fits in an exact sequence 
$$1\rightarrow \Lambda_n \rightarrow B_n(\mathbb{R}P^2) \rightarrow Q_8 \rightarrow 1$$
where $\Lambda_n$ is SPF, for all $n>3$. Hence $B_n(\mathbb{R}P^2)$ satisfies FIC by Corollary \ref{finitebyspf}.
\end{proof}

\begin{theorem}
The full braid groups $FB_n(\mathbb{R}P^2)$ satisfy FIC for all $n>0$.
\end{theorem}
\begin{proof}
In \cite{buskirk} it is proven that $FB_1(\mathbb{R}P^2)=\mathbb{Z}_2$ and $FB_2(\mathbb{R}P^2)$ is isomorphic to a dicyclic group of order 16. Hence $FB_1(\mathbb{R}P^2)$ and $FB_2(\mathbb{R}P^2)$ satisfy FIC because they are finite.
In \cite{silvia} Theorem 29 it is proven that $FB_n(\mathbb{R}P^2)$ fits in an exact sequence
$$1\rightarrow S_n \rightarrow FB_n(\mathbb{R}P^2) \rightarrow FB_n(\mathbb{R}P^2)/S \rightarrow 1$$
where $S_n$ is a normal SPF subgroup of $FB_n(\mathbb{R}P^2)$ with finite index, for all $n>2$. Hence by Theorem \ref{finitebyspf} we conclude that $FB_n(\mathbb{R}P^2)$ satisfies FIC for all $n>2$.
\end{proof}

Recall that if the groups $G$ and $H$ satisfy FIC then $F\times H$ also satisfies FIC. Therefore we have the following 

\begin{theorem}
  Let $G$ be a braid group of a surface in any number of strands then $G\times \dbZ^n$ satisfies FIC for all $n\geq 1$.
\end{theorem} 

%Let $M$ be a closed aspherical $n$-manifold, the Borel Conjecture states that if $f:M\to N$ is a homotopy equivalence, where $N$ %is a closed $n$-manifold, then $f$ is homotopic to a homeomorphism. The validity of the FIC for $K$ and $L$ theory implies the %Borel Conjecture in high dimensions, see \cite{luckKL}. From the results in this paper we have:

%\begin{theorem}
%Let $M$ be a closed  aspherical $n$-manifold, $n\geq 5$. Assume $\pi_1(M)$ is isomorphic to a braid group of a surface. Then, the %Borel Conjecture holds for $M$.
%
%\end{theorem}

\bibliographystyle{alpha}
\bibliography{myblib}
\end{document}